\documentclass [11pt]{amsart}
\usepackage{amscd, amsmath,amsthm,amssymb, amsfonts, epsfig}
\usepackage{bbm}
\usepackage{graphicx}
\usepackage[utf8]{inputenc}
\usepackage{xcolor}
\usepackage{mathrsfs}  
\usepackage{esvect}
\usepackage{enumitem}
\usepackage{indentfirst}
\usepackage{caption}
\usepackage{subfigure,tikz}
\usepackage{tkz-euclide}

\usetikzlibrary{intersections,positioning, calc}

\tikzset{
  dim above/.style={to path={\pgfextra{
        \pgfinterruptpath
        \draw[>=latex,|<->|] let
        \p1=($(\tikztostart)!2mm!90:(\tikztotarget)$),
        \p2=($(\tikztotarget)!2mm!-90:(\tikztostart)$)
        in(\p1) -- (\p2) node[pos=.5,sloped,above]{#1};
        \endpgfinterruptpath
      }(\tikztostart) -- (\tikztotarget) \tikztonodes
    }}, 
    dim below/.style={to path={\pgfextra{
        \pgfinterruptpath
        \draw[>=latex,|<->|] let
        \p1=($(\tikztostart)!2mm!-90:(\tikztotarget)$),
        \p2=($(\tikztotarget)!2mm!90:(\tikztostart)$)
        in(\p1) -- (\p2) node[pos=.5,sloped,above]{#1};
        \endpgfinterruptpath
      }(\tikztostart) -- (\tikztotarget) \tikztonodes
    }},}

\newtheorem{theorem}{Theorem}
\newtheorem{corollary}{Corollary}
\newtheorem{lemma}{Lemma}

\newtheorem{definition}{Definition}
\newtheorem{proposition}{Proposition}

\theoremstyle{remark}
\newtheorem*{remark}{Remark}

\begin{document}
\title[A shape optimization problem for the first eigenvalue]{A shape optimization problem for the first mixed Steklov-Dirichlet eigenvalue}

\author{Dong-Hwi Seo}
\address{Department of Mathematical Sciences, KAIST, 291 Daehak-ro Yuseong-gu, Daejeon
34141, Korea}
\email{donghwi.seo@kaist.ac.kr}


\maketitle

\begin{abstract}
We consider a shape optimization problem for the first mixed Steklov-Dirichlet eigenvalues of domains bounded by two balls in two-point homogeneous space. We give a geometric proof which is motivated by Newton's shell theorem..

\end{abstract}

\section{Introduction}
\label{intro}
Let $M^m$ be a Riemannian manifold of dimension $m\ge 2$ and $\Omega\subset M$ a bounded smooth domain with the boundary $\partial \Omega$. Let $\partial \Omega = C_1 \cup C_2$ where $C_1$ and $C_2$ are disjoint components. A mixed Steklov-Dirichlet eigenvalue problem is to find $\sigma \in \mathbb{R}$ for which there exists $u\in C^{\infty}(\Omega)$ satisfying
\begin{align} \label{problem}
\left\{
\begin{array}{rcll}
     \Delta u &=& 0&   \text{in   } \Omega  \\
     u &=& 0& \text{on   } C_1\\
     \frac{\partial u}{\partial \eta} &=& \sigma u &\text{on   } C_2
\end{array} \right.
,
\end{align}
where $\eta$ is the outward unit normal vector along $C_2$. When $C_1 = \phi$, the problem becomes the Steklov eigenvalue problem introduced by Steklov in 1902 \cite{stekloff1902problemes}. We will find a domain maximizing the lowest $\sigma$ in a class of subsets in $M$. We call this problem by a shape optimization problem of the first eigenvalue.

The shape optimization problem of the first nonzero Steklov eigenvalue in Euclidean space has been studied since the 1950s. In 1954, Weinstock considered the case when $M=\mathbb{R}^2$ \cite{weinstock1954inequalities}. He showed that the disk is the maximizer among all the simply connected domains with the same boundary lengths. Recently, Bucur, Ferone, Nitsch, and Trombetti studied this perimeter constraint shape optimization problem in any dimension among all the convex sets, and showed that the ball is the maximizer \cite{bucur2017weinstock}. Without the convexity condition, Fraser and Schoen proved the ball cannot be a maximizer even among all the smooth contractible domains of fixed boundary volume in $\mathbb{R}^m$, $m\ge 3$ \cite{fraser2019shape}. On the other hand, Brock \cite{brock2001isoperimetric} proved in 2001 that the ball is the maximizer among all the smooth domains with fixed domain volume in $\mathbb{R}^m$, $m \ge 2$. Note that he does not need any topological restriction.

These shape optimization problems have been extended to non-Euclidean spaces as well. The first result in this direction was given by Escobar \cite{escobar1999isoperimetric} who showed that the first nonzero eigenvalue is maximal for the geodesic disk among all the simply connected domains with fixed domain area in simply connected complete surface $M^2$ with constant Gaussian curvature. In 2014, Binoy and Santhanam extended this result to noncompact rank one symmetric spaces of any dimension \cite{binoy2014sharp}.

Regarding mixed Steklov-Dirichlet eigenvalue problems, it was considered by Hersch and Payne in 1968 \cite{hersch1968extremal}. They considered the problem (\ref{problem}) when $\Omega\subset \mathbb{R}^2$ is a doubly connected region bounded by the inner and the outer boundaries, $C_1$ and $C_2$, respectively. Then among all the conformally equivalent domains with fixed perimeter of $C_2$, the annulus bounded by two concentric circles is the maximizer. Recently, Santhanam and Verma considered connected regions in $\mathbb{R}^m$ with $m \ge 3$ that are bounded by two spheres of given radii and gave the Dirichlet condition only on the inner sphere. Then the maximizer is obtained by the domain bounded by two concentric spheres \cite{verma2018eigenvalue} (See also Ftouhi \cite{ftouhi2019where}).

The aim of this paper is to extend Santhanam and Verma's result \cite{verma2018eigenvalue} from Euclidean spaces to two-point homogeneous spaces including $\mathbb{R}^2$.  The main theorem is as follows. We denote the injectivity radius of $M$ and the closure of a set $A\subset M$ by $inj(M)$ and $\text{cl}(A)$, respectively.

\begin{theorem} \label{main} Let $M$ be a two-point homogeneous space. Let $\mathbf{B}_1$ and $\mathbf{B}_2^\prime$ be geodesic balls of radii $R_1, R_2>0$, respectively, such that $\textnormal{cl}(\mathbf{B}_1)\subset \mathbf{B}_2^\prime$ and $R_2 < inj(M)/2$. Then the first mixed Steklov-Dirichlet eigenvalue of the problem  
\begin{align} \label{vermaproblem}
    \left\{
    \begin{array}{rcll}
    \Delta u &=& 0 &\textnormal{in   } \mathbf{B}_2^\prime \symbol{92} \textnormal{cl}(\mathbf{B}_1)\\
    u &=& 0 &\textnormal{on   } \partial \mathbf{B}_1\\
    \frac{\partial u}{\partial \eta} &=& \sigma u &\textnormal{on   } \partial \mathbf{B}_2^\prime
    \end{array}
    \right.
\end{align}
($\eta$ : the outward unit normal vector along $\partial \mathbf{B}_2^\prime$) attains maximum if and only if $\mathbf{B}_1$ and $\mathbf{B}_2^\prime$ are concentric.
\end{theorem}
Two-point homogeneous space has similar geometric properties with Euclidean space. For example, for two geodesic balls $\mathbf{B}_3$ and $\mathbf{B}_4'$ of radii $R_1$ and $R_2$, respectively, satisfying $\text{cl}(\mathbf{B}_3) \subset \mathbf{B}_4'$, $\mathbf{B}_4' \setminus \text{cl}(\mathbf{B}_3)$ is isometric to $\mathbf{B}_2' \setminus \text{cl}(\mathbf{B}_1)$ if and only if the distance of the centers of $\mathbf{B}_3$ and $\mathbf{B}_4'$ is equal to that of $\mathbf{B}_1$ and $\mathbf{B}_2'$. Furthermore, using additional angles, which are not usual Riemannian angles, there are laws of trigonometries and conditions for triangle conditions (for example, see Proposition \ref{congruence}) in two-point homogeneous space.

In section \ref{section2}, we will briefly review the variational characterization of the mixed Steklov-Dirichlet eigenvalue problem (\ref{vermaproblem}) (section \ref{section2.1})  as well as two-point homogeneous spaces and its trigonometries (section \ref{section2.2}).

Section \ref{section3}  is devoted to the proof of the main theorem. We estimate the first eigenvalue by substituting an appropriate test function in the Rayleigh quotient for the variational characterization of the first eigenvalue for $\mathbf{B}_2'\setminus \textnormal{cl}(\mathbf{B}_1)$  (see (\ref{variational characterization})). Our test function is the first mixed Steklov-Dirichlet eigenfunction on the domain bounded by concentric balls (Proposition \ref{firsteigenfunction}).

Before proving main theorem, we obtain some crucial lemmas in section \ref{section3.2.1}. As a corollary, we give a proof of Newton's shell theorem for a ball whose radius less than $inj(M)/2$ in a two-point homogeneous space (see Corollary \ref{newton} and the following Remark). This theorem was first proved by Newton \cite{Newton1972principia} for $M=\mathbb{R}^3$ and it was extended to constant curvature spaces by Kozlov \cite{Kozlov2000newton} and Izmestiev and Tabachnikov \cite{Izmestiev2017ivory}.

In section \ref{section3.2.2}, we prove the main theorem for noncompact rank one symmetric spaces (noted nCROSSs). We estimate the denominator of the Rayleigh quotient by using a geometric proof motivated by the proof of Newton's shell theorem (see Corollary \ref{L2sum}).

However, the argument of estimation of the numerator of the Rayleigh quotient in section \ref{section3.2.2} does not work for compact rank one symmetric spaces (noted CROSSs) when $inj(M)/4\le R_2< inj(M)/2$. We overcome this problem by partitioning domains and observing symmetry of a sphere (see section \ref{section3.2.3}). It is reminiscent of Ashbaugh and Benguria's work \cite{Ashbaugh:1995:SUB} on a shape optimization problem of the first nonzero Neumann eigenvalue for bounded domains with fixed domain volume in $M=\mathbb{S}^m$. They showed that the geodesic ball is the maximizer if we restricted $\Omega$ to be contained in a geodesic ball of radius $inj(M)/2$. This domain restriction is a refinement of Chavel's work \cite{Chavel:1980:LEI}. But the analogous result is not known if $M$ is CROSS (see Aithal and Santhanam \cite{Aithal:1996:SUB}).

\section{Background} \label{section2}
\subsection{The eigenvalue problem}\label{section2.1}
A mixed Steklov-Dirichlet eigenvalue problem (\ref{problem}) is equivalent to the eigenvalue problem of the Dirichlet-to-Neumann operator :
\begin{align*}
    L : C^{\infty}(C_2) &\longrightarrow C^{\infty}(C_2)\\
    u &\longmapsto \frac{\partial \hat{u}}{\partial \eta},
\end{align*}
where $\hat{u}$ is the harmonic extension of $u$ satisfying the following
\begin{align*}
    \left\{
\begin{array}{rcll}
     \Delta \hat{u} &=& 0&   \text{in   } \Omega  \\
     \hat{u} &=& 0& \text{on   } C_1\\
     \hat{u} &=& u &\text{on   } C_2
\end{array} \right. .
\end{align*}
Then $L$ is a positive-definite, self-adjoint operator with discrete spectrum (see for instance \cite{agranovich2006mixed}),
\begin{align*}
    0<\sigma_1(\Omega)\le \sigma_2(\Omega) \le \cdots \rightarrow \infty,
\end{align*}
provided that $C_1\neq \phi$.
We call $\sigma_k(\Omega)$ by the $k$th mixed Steklov-Dirichlet eigenvalue, or simply the $k$th eigenvalue. An eigenfunction of $L$ corresponding to $\sigma_k(\Omega)$ is called the $k$th mixed Steklov-Dirichlet eigenfunction, or the $k$th eigenfunction. Then the first eigenvalue $\sigma_1(\Omega)$ is characterized variationally as follows
\begin{align} \label{variational characterization}
    \sigma_1(\Omega) = \inf \left\{ \left.\frac{\displaystyle \int_{\Omega}|\nabla v|^2 dV}{\displaystyle\int_{C_2}v^2 ds} \right| v \in H^1(\Omega)\setminus\{0\} \textnormal{ and } v=0 \textnormal{ on } C_1   \right\}.
\end{align}
For convenience we shall call the harmonic extension of the $k$th eigenfunction by the $k$th mixed Steklov-Dirichlet eigenfunction or the $k$th eigenfunction.

\subsection{Two-point homogeneous spaces and triangle congruence conditions}\label{section2.2}
Three points in a Euclidean space determine a triangle when three points are not lie on a single line. In classical geometry, there are several congruence conditions on triangles and it is determined by lengths of sides and angles. For example, side-angle-side (SAS) congruence is given by two side lengths and the included angle. In two-point homogeneous spaces, analogous properties also hold with additional angles. These facts are obtained by the laws of trigonometries. In this section, we give some information about two-point homogeneous spaces and its congruence conditions of triangles which will be used later. See \cite{wolf1972spaces},\cite{hsiang1989laws},\cite{brehm1990shape} for more details.

\begin{definition}
A connected Riemannian manifold $M$ is called two-point homogeneous space if $x_i,y_i \in M, i=1,2$ with $dist(x_1, y_1)=dist(x_2,y_2)$, there is an isometry $g$ of $M$ such that $g(x_1)=x_2$ and $g(y_1)=y_2$.
\end{definition}
In fact, two-point homogeneous spaces are Euclidean spaces or rank one symmetric spaces. We will call the latter spaces  by ROSSs. Furthermore, compact ROSS and noncompact ROSS are denoted by CROSS and nCROSS, respectively. Then two-point homogeneous spaces with their isotropy representations are classified as in the Table \ref{twopointhomogeneous space} (see \cite{wolf1972spaces},\cite{hsiang1989laws}). Here $m\ge 1, n\ge 2$ and $m=\textnormal{dim}_{\mathbb{R}}M = n \cdot \textnormal{dim}_{\mathbb{R}}\mathbb{K}$.

\begin{table}[]
\caption{Two-point homogeneous spaces, $m\ge 1, n\ge 2$.}\label{twopointhomogeneous space}
\begin{center}
\begin{tabular}{|c|c|c|c|c|}
\hline
 &   & CROSS & nCROSS & Isotropy representation \\
 \hline
$\mathbb{K}=\mathbb{R}$ & $\mathbb{R}^m$ & $\mathbb{S}^m, \mathbb{R}P^n$  & $\mathbb{R}H^n$ & $(O(m),\mathbb{R}^m)$ \\
\hline
$\mathbb{K}=\mathbb{C}$ & $\cdot$ & $\mathbb{C}P^n$ & $\mathbb{C}H^n$ &  $(U(n),\mathbb{R}^{2n})$ \\
\hline
$\mathbb{K}=\mathbb{H}$ & $\cdot$ & $\mathbb{H}P^n$ & $\mathbb{H}H^n$ & $(Sp(1)\times Sp(n),\mathbb{R}^{4n})$ \\
\hline
$\mathbb{K}=\mathbb{O}$ & $\cdot$ & $\mathbb{O}P^2$ & $\mathbb{O}H^2$ & $(Spin(9),\mathbb{R}^{16})$\\
\hline
\end{tabular}
\end{center}
\end{table}

 An angle is given by two directions at a point $P$. It is classified by its congruence classes which are given by the orbit space of $U_{P}M \times U_{P}M/K$, where $U_PM$ is the unit sphere in the tangent space of $M$ at $P$, and $K$ is the isotropy subgroup of the isometry group $M$ at $P$. The orbit space can be seen by fixing the first component by the action of $K$. More precisely, it is equivalent to an orbit space $U_PM/H$ of an isotropy group $H\subset K$ with respect to a point in $U_PM$. Then it can be checked that for given $\vec{v}_1\in U_PM$,  $H$-invariant subspaces are $\mathbb{R}\cdot \vec{v}_1, \mathbb{K}'\cdot \vec{v}_1,$ and the subspace orthogonal to $\mathbb{K}\cdot \vec{v}_1$, where $\mathbb{K}=\mathbb{R}, \mathbb{C}, \mathbb{H}$, and $\mathbb{O}$ and $\mathbb{K}'$ is the set of pure imaginary numbers in $\mathbb{K}$. Then a direction $\vec{v}_2$ is determined up to $H$-action by the following angular invariants (for more details, see \cite{hsiang1989laws},\cite{brehm1990shape}): 
 \begin{itemize}
     \item $\lambda(\vec{v}_1,\vec{v}_2) = \angle(\vec{v}_1,\vec{v}_2)$ ; $0\le\lambda\le\pi$,
     \item $\varphi(\vec{v}_1,\vec{v}_2)=\angle(\vec{v}_1, \mathbb{K}\cdot \vec{v}_2)$ ; $0\le\varphi\le\frac{\pi}{2}$,
 \end{itemize}
where $\angle(\vec{v}_1, \vec{v}_2)$ is the usual (Riemannian) angle and $\angle(\vec{v}_1,\mathbb{K}\cdot \vec{v}_2)$ is the angle between $\vec{v}_1$ and the subspace $\mathbb{K}\cdot \vec{v}_2$. Note that when $\mathbb{K}=\mathbb{R}$, $\lambda=\varphi$ or $\lambda=\pi-\varphi$. Then angular invariants satisfy following relations : 
\begin{align} \label{180}
    \lambda(\vec{v}_1,-\vec{v}_2) &= \pi - \lambda(\vec{v}_1,\vec{v}_2), \\ \label{180_2}
    \varphi(\vec{v}_1,-\vec{v}_2) &= \varphi(\vec{v}_1,\vec{v}_2).
\end{align}

Using the previous $H$-invariant decomposition, we can write the metric of ROSS $M$ explicitly.  Let $s(r)$ and $c(r)$ be functions defined as follows :
\begin{align*}
    s(r)=
    \begin{cases}
    \sin{r} \text{ with } 0\le r<\pi & \text{if } M=\mathbb{S}^m\\
    \sin{r} \text{ with } 0\le r<\frac{\pi}{2} & \text{if } M=\mathbb{R}P^n, \mathbb{C}P^n, \mathbb{Q}P^n, \mathbb{O}P^2\\
    \sinh{r} & \text{if } M \text{ is nCROSS}
    \end{cases}\\
\end{align*}
and
\begin{align*}
    c(r)=
    \begin{cases}
    \cos{r} \text{ with } 0\le r<\frac{\pi}{2} & \text{if } M=\mathbb{C}P^n, \mathbb{Q}P^n, \mathbb{O}P^2\\
    \cosh{r} & \text{if } M \text{ is nCROSS}.
    \end{cases}
    \end{align*}
Then the metric $(ds)^2$ is given by
\begin{align}\label{metric}
  (ds)^2 = (dr)^2+(s(r))^2(c(r))^2g+(s(r))^2h,  
\end{align}
where $(dr)^2,g,$ and $h$ are written by $\sigma_1^2$ with the coframe $\sigma_1$ dual to $\vec{v}_1$;  $\sigma_2^2+\cdots +\sigma_k^2$ with coframes $\sigma_2,\dots, \sigma_k$ dual to orthonormal basis of $\mathbb{K}'\cdot \vec{v}_1$; $\sigma_{k+1}^2+\cdots +\sigma_m^2$ with coframes $\sigma_{k+1},\dots, \sigma_m$ dual to the complement orthonormal basis of $\mathbb{R}^m$. Since the density function $\omega$ only depends on distance, we may define $\omega$ as a one-variable function
\[ \omega(r) = (s(r))^{m-1}(c(r))^{k-1}.
\]
Then the sectional curvature $K_M$ of $M$: 
\begin{align} \label{curvature}
    \begin{cases}
        1\le K_M \le 4 & \text{if } M \text{is CROSS}\\
        -4\le K_M \le -1 & \text{if }M \text{is nCROSS}.
    \end{cases}
\end{align}
In particular, $\mathbb{S}^m$ and $\mathbb{R}P^n$ has sectional curvature 1. Then  the condition $0<R_2 < \frac{inj(M)}{2}$ in Theorem 1 implies:
\begin{align}
\left\{
    \begin{array}{ll}
     0<R_2<\frac{\pi}{2} & \text{if } M=\mathbb{S}^m  \\
     0<R_2<\frac{\pi}{4} & \text{if } M=\mathbb{R}P^n, \mathbb{C}P^n, \mathbb{H}P^n, \mathbb{O}P^2\\
     0<R_2 & \text{otherwise}.
    \end{array}
\right.
\end{align}

Now consider a triangle $(PQR)$ in $M$ with the metric (\ref{metric}), which consists of three distinct points $P,Q,R\in M$ and three connecting geodesics $QR, RP, PQ$. The side lengths will be denoted by $p, q$, and $r$, respectively and the two angular invariants $\lambda, \varphi$ determined by the two tangent vectors of geodesic rays $\vv{PQ}$ and $\vv{PR}$ at $P$ will be denoted by $\lambda(P)$ and $\varphi(P)$, respectively. Furthermore we can denote $\lambda(Q)$, $\varphi(Q)$, $\lambda(R)$, and $\varphi(R)$ in an analogous way. Then it is known that there are congruent conditions of triangles. We introduce some conditions which will be used later. For more conditions, see \cite{brehm1990shape}.
\begin{proposition}\label{congruence}
 A triangle (PQR) in ROSS with the metric $\textnormal{(\ref{metric})}$ is uniquely determined up to isometry as follows :
 \begin{enumerate}[label=(\alph*)]
    \item \label{a} $p, q,$ and $\lambda(P)$ with $0<p,q,r<\pi$ and $q<p<\frac{\pi}{2}$ if $M$ is $\mathbb{S}^m$.
    \item \label{b} $p, q,$ and $\lambda(P)$ with $0<p,q,r<\frac{\pi}{2}$ and $q<p<\frac{\pi}{4}$ if $M$ is $\mathbb{R}P^n$.
     \item \label{c} $p, q,$ $\lambda(P)$, and $\varphi(P)$ with $0<p,q,r<\frac{\pi}{2}$ and $(p-q)(\cos{p}-\sin{q}\cos{\varphi(P)})>0$ if $M$ is $\mathbb{C}P^n, \mathbb{H}P^n$ or $\mathbb{O}P^2$.  
     \item \label{d} $p, q, \lambda(P)$, and $\varphi(P)$ with $0<p,q,r$ and $q<p$ if $M$ is nCROSS.
 \end{enumerate}
\end{proposition}
\begin{proof}
    The proof of  \ref{a} can be found in Section VI in \cite{todhunter1863spherical}. In fact, the condition $p<\frac{\pi}{2}$ can be replaced by $p+q < \pi$. The proof of \ref{b} follows from \ref{a}.  The proofs of \ref{c} and \ref{d} can be found in (ix) and (ix') of Theorem 4 and 4' in \cite{brehm1990shape}. \qed
\end{proof}

\section{Main proof} \label{section3}
Let $M$ be a ROSS with the metric (\ref{metric}). Let $C$ and $C'$ be the centers of $\mathbf{B}_1$ and $\mathbf{B}_2^\prime$, respectively. Define $\mathbf{B}_2$ to be the ball of radius $R_2$, centered at $C$. See Fig. \ref{balls}.
\begin{figure}
       \begin{center}
    \begin{tikzpicture}
\coordinate  (C') at (0,0);
\coordinate (D) at (2,0);
\coordinate (C) at (-1,0);
\fill[gray!40,even odd rule] (C) circle (0.7) (C') circle (2);

\draw[name path =circle](C') circle(2cm);
\draw (C) circle (0.7cm);
\draw (C') -- (1,1.7);
\draw (C) --(-1.7,0);
\node[left, scale=0.8] at (0.6, 1.02){$R_2$};
\node[above, scale =0.8] at (-1.3,0){$R_1$};
\node at (-1,1){$\mathbf{B}_1$};
\node at (0,2.25){$\mathbf{B}_2^\prime$};
\node[below] at (-1,0){$C$};
\node[right] at (0,0) {$C'$};
\foreach \point in {C,C'}
	\fill [black] (\point) circle (1.5pt);

	\end{tikzpicture}
\qquad	
	 \begin{tikzpicture}

\coordinate  (C) at (0,0);

\fill[gray!40,even odd rule] (C) circle (0.7) (C) circle (2);	

\draw[name path =circle](C) circle(2cm);
\draw (C) circle (0.7cm);

\node at (0,2.25){$\mathbf{B}_2$};
\node at (0,0.95){$\mathbf{B}_1$};
\node[above] at (0,0) {$C$};
\foreach \point in {C}
	\fill [black] (\point) circle (1.5pt);

	\end{tikzpicture}
	\end{center}
    \caption{Description of $\mathbf{B}_1$ and $\mathbf{B}_2^\prime$ (left), and $\mathbf{B}_1$ and $\mathbf{B}_2$ (right)}
    \label{balls}
\end{figure}
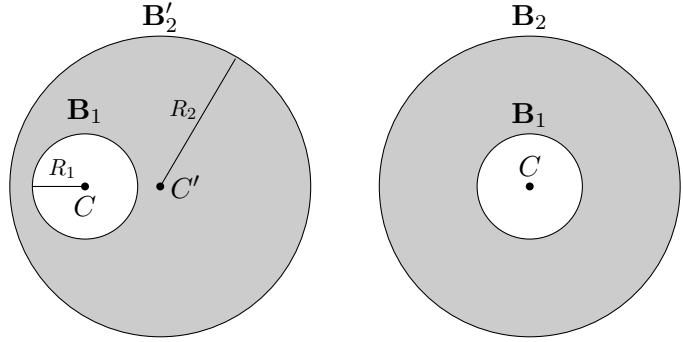
\subsection{The first eigenfunctions} \label{section3.1}
In this section, we derive an explicit formula for the first mixed Steklov-Dirichlet eigenfunctions in $\mathbf{B}_2\setminus \textnormal{cl}(\mathbf{B}_1)$. Using the following standard argument as in  \cite{castillon2019spectral} and \cite{verma2018eigenvalue}, we can show that the first eigenfunction is a function that only depends on the distance from $X$.

Using seperation of variables, a mixed Steklov-Dirichlet eigenfunction $u(r,\theta_1,\dots \theta_{m-1})$ in $\mathbf{B}_2\setminus \textnormal{cl}(\mathbf{B}_1)$ is obtained by multiplying a Laplacian eigenfunction $f(\theta_1,\dots, \theta_{m-1})$ on $\partial \mathbf{B}_2$ by an appropriate radial function $a(r)$. Here, $(r, \theta_1,\dots \theta_{m-1})$ is the polar coordinate in $T_CM$. More precisely, we have the following lemma.

\begin{lemma}\label{a(r)}
For given Laplacian eigenfunction $f : \mathbb{S}^{m-1} \rightarrow \mathbb{R}$, there exists a nonnegative function $a:[R_1, \infty)\rightarrow \mathbb{R}^+\cup\{0\}$ such that the function $u(r,\theta_1,\dots \theta_{m-1})=a(r)f(\theta_1,\dots,\theta_{m-1})$ is harmonic and $a(R_1)=0$. Specifically, $u$ is a mixed Steklov-Dirichlet eigenfunction on $\mathbf{B}_2\setminus \textnormal{cl}(\mathbf{B}_1)$.
\end{lemma}
\begin{proof}
    Let $a: [R_1, \infty)\rightarrow \mathbb{R}$ be a smooth function. Then we have  
    \begin{align}
         \label{laplacian}
    \Delta(a(r)f(\theta_1,\dots \theta_{m-1})) = (a''(r)+(m-1)h(r)a'(r)-\lambda(S(r))a(r))f, 
    \end{align}
    where $S(r)$ is the geodesic sphere of radius $r$, centered at $C$, $\lambda(S(r))$ is the eigenvalue of $f$ on $S(r)$, and $h(r)$ is the mean curvature of the geodesic sphere $S(r)$ that is obtained by trace of the shape operator of $S(r)$ with respect to the inner normal vector times $\frac{1}{m-1}$. Since $(m-1)rh$ and $r^2\lambda(S(r))$ is analytic (see \cite[Proposition 5.3]{bergery1982laplacians}), 0 is a regular singular point and we can find two linearly independent solutions of the equation (see \cite[Theorem 12.1 on p.85]{hartman2002ordinary}). Thus there exists $a(r)$ such that $a(R_1)=0$ and $a\cdot f$ is harmonic. Note that $\lambda(S(r))$ is nonnegative. Then from (\ref{laplacian}) with maximum principle, $a(r)$ is a not sign-changing function. Thus we may assume $a(r)$ is nonnegative in $[R_1,\infty)$.  \qed
\end{proof}

Since Laplace eigenfunctions on $\mathbb{S}^{m-1}$ are indeed Laplace eigenfunctions on $\partial\mathbf{B}_2$ (see \cite[Theorem 3.1]{castillon2019spectral}, or \cite[Corollary 5.5]{bergery1982laplacians}) and it consists of a basis of $L^2(\partial\mathbf{B}_2)$, our mixed Steklov-Dirichlet eigenfunctions restrict to $\partial \mathbf{B}_2$ become a basis of $L^2(\partial \mathbf{B}_2)$. It implies every mixed Steklov-Dirichlet eigenfunction is written by a product of a Laplacian eigenfunction and a radial function.  Then some computations as in \cite[Section 2.1]{verma2018eigenvalue}, we can show that the first mixed Steklov-Dirichlet eigenfunction is corresponding to the first Laplacian eigenfunction as the following lemma. Here we count constant function as the first Laplacian eigenfunction on $\mathbb{S}^{m-1}$.

\begin{lemma}
Let $a_1 : [R_1, \infty) \rightarrow \mathbb{R}^+ \cup \{0\}$ be the nonnegative function that obtained from Lemma \ref{a(r)} when the given Laplacian eigenfunction is constant. Then we have
\begin{align*}
    \frac{a_1'(R_2)}{a_1(R_2)} \le \frac{a'(R_2)}{a(R_2)},
\end{align*}
and the equality holds if and only if $f$ is constant. Here we used notations $a,f$ in Lemma \ref{a(r)}.
\end{lemma}
\begin{proof}
    From the harmonicity of mixed Steklov-Dirichlet eigenfunctions we obtained the following equations.
\begin{align}\label{har1}
    a_1''(r) + \frac{\omega'(r)}{\omega(r)}a_1'(r)&=0
\end{align}
\begin{align}\label{har2}
    a''(r)+ \frac{\omega'(r)}{\omega(r)}a'(r)-\lambda(S(r))a(r)&=0.
\end{align}
Then $((\ref{har1})\times a - (\ref{har2})\times a_1 )\times \omega(r)$ implies
\begin{align*}
    0&= \omega a_1''a-a_1a''+\omega'a_1'a-\omega'a_1a'+\lambda(S(r))\omega a_1a \\
    &\ge \omega a_1''a-a_1a''+\omega'a_1'a-\omega'a_1a' \\
    &= (\omega(a_1a'-a_1a'))'.
\end{align*}
Note that the equality of the inequality holds if and only if $\lambda(S(r))=0$ that is $f$ is constant.
Since $a_1(R_1)=a(R_1)=0$, we have 
\begin{align*}
    a(R_2)a_1'(R_2)-a'(R_2)a_1(R_2) \le 0,
\end{align*}
or
\begin{align*}
    \frac{a'(R_2)}{a(R_2)}\le \frac{a_1'(R_2)}{a_1(R_2)}.
\end{align*}\qed
\end{proof}

We can easily observe that
\begin{align*}
    \frac{a'(R_2)}{a(R_2)}
\end{align*}
is the Steklov eigenvalue corresponding to $u$. Now we can compute the first mixed Steklov-Diriclet eigenfunction on $\mathbf{B}_2\setminus \text{cl}(\mathbf{B}_1)$ as follows. 

Since the first Laplacian eigenfunctions are constants, we obtain the following. We abuse notation $a_1$ by $a$ for convenience.

\begin{proposition} \label{firsteigenfunction}
    Let $r_X : \mathbf{B}_2'\rightarrow [0,\infty)$ be the distance function from $X$. Let $a : [R_1, \infty) \rightarrow \mathbb{R}$ be a function defined by 
    \begin{align*}
        a(r) = \int_{R_1}^r \frac{1}{\omega(t)}dt.
    \end{align*}
    Then the first mixed Steklov-Dirichlet eigenfunction in $\mathbf{B}_2\setminus \textnormal{cl}(\mathbf{B}_1)$ is $a\circ r_C$ up to constant.
\end{proposition}
\begin{proof}
    By the argument in the paragraph, the first eigenfunction can be written by 
    \begin{align*}
        a\circ r_C,
    \end{align*}
    where $a:[R_1, \infty) \rightarrow \mathbb{R}$ is a real-valued function. Then, the harmonicity of the eigenfunction implies 
    \begin{align*}
        0=\Delta a(r) = a^{\prime\prime}(r)+\frac{\omega^\prime(r)}{\omega(r)}a^\prime(r) = \frac{1}{\omega(r)}(a^\prime(r)\omega(r))^\prime.
    \end{align*}
    Here, we used $r$ instead of $r_C$ for simplicity of notation. With the fact that $a(R_1)=0$ from the boundary condition, we obtain the formula of $a(r)$ up to constant. \qed
\end{proof}

\subsection{Crucial lemmas and the proof for nCROSS} \label{section3.2}
We begin with two definitions (see Fig. \ref{descriptionofdirectionandangle}).
\begin{definition}
    For given $X\in \mathbf{B}_2^\prime$, a vector-valued function $\vec{v}_X: M\setminus{\{X\}}\rightarrow T_XM$ is defined by $P\in M\setminus \{X\}$ and $\vec{v}_X(P)\in T_XM$ such that $\vec{v}_X(P)$ is the unit tangent vector of the geodesic ray $\vv{XP}$ at $X$.   
\end{definition}

\begin{definition}
    For given $X\in \mathbf{B}_2^\prime$, an angle function with respect to $X$, $\lambda_X : \partial \mathbf{B}_2^\prime \rightarrow [0,\pi]$, is a map that assigns to each $P \in \partial \mathbf{B}_2^\prime$ an angle $\lambda(P)$ of the triangle (PXC).
\end{definition}

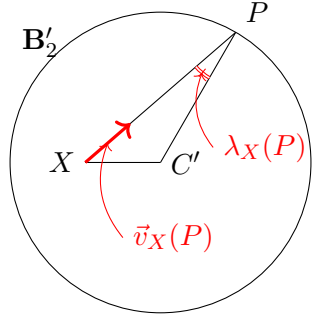
\begin{figure}
    \begin{center}
        \begin{tikzpicture}
        \coordinate [label=right : $C'$] (C') at (0,0);
        \coordinate (D) at (2,0);
        \coordinate (X) at (-1,0);
        \node [left] at (X) {$X$};
        \draw[name path =circle](C') circle(2);
        \coordinate  [label=above right : $P$] (P) at ($(C)!1!60:(D)$);
        \draw (P)--(C'); 
        \draw (X)--(C');
        \draw (X)--(P);
        \tkzMarkAngle[size=0.7,double, color=red](X,P,C');
        \coordinate (P_s) at ($(X)!0.3!(P)$);
        \draw [->, very thick, red] (X) -- (P_s);
        \node at (-1.6,1.6){$\mathbf{B}_2^\prime$};
        
        {\color{red}
        \path[->] (-0.5,-1) edge [bend left] ($(X)!0.5!(P_s)$);
        \node[right] at (-0.5,-1) {$\vec{v}_X(P)$};
        \path[->]  (0.7, 0.2) edge [bend left] (0.5542,1.1923);
        \node[right] at (0.7, 0.2) {$\lambda_X(P)$};
        }
        \end{tikzpicture}
    \end{center}
    \caption{Description of $\vec{v}_X(P)$ and $\lambda_X(P)$}
    \label{descriptionofdirectionandangle}
\end{figure}

For a given parametrization of $M$ around $X$, we can identify $T_XM$ with $\mathbb{R}^m$. Then we can give the following definition.
\begin{definition}
    For given $X\in \mathbf{B}_2^\prime$ and a parametrization of $M$ around $X$, a map $\pi_X : \mathbb{S}^{m-1}\subset T_XM \rightarrow \partial \mathbf{B}_2^\prime$ is defined by $\pi_X(v) = \textnormal{exp}_X([0,\infty)\cdot v) \cap \partial \mathbf{B}_2^\prime$, i.e. $\pi_X(v)$ is the point of $\partial \mathbf{B}_2'$ in the geodesic emanating from $X$ in $v$ direction.
\end{definition}

Note that $\pi_X$ has the inverse map. Thus, for any $P\in \partial \mathbf{B}_2^\prime$, we can find $p_s\in \mathbb{S}^{m-1}$ such that $P=\pi_X(p_s)$. Let $c'_s\in\mathbb{S}^{m-1}$ such that the geodesic ray $\text{exp}_X([0,\infty)\cdot {c_s'})$ passes through $C'$. Then we can define $-p_s \in \mathbb{S}^{m-1}$ such that it is the symmetric points of $p_s$ with respect to $X$. We define $\bar{p_s}\in \mathbb{S}^{m-1}$ such that it is the symmetric point of $p_s$ with respect to the line passing through $X$ and $c_s'$ in the plane spanned by the vectors $p_s$ and $c_s'$. In addition, $-\bar{p_s}$ can be defined as the symmetric point of $\bar{p_s}$ with respect to $X$. Now we denote $\text{exp}_X(-p_s)$, $\text{exp}_X(\bar{p_s})$, and $\text{exp}_X(-\bar{p_s})$ by $-P$, $\bar{P}$, and $-\bar{P}$, respectively. Fig. \ref{points} explains the situation.

\begin{figure}[hbt!]
\centering
    \hspace*{-1.5cm}
    \begin{tikzpicture}
    \coordinate [label=right : $C'$] (C') at (0,0);
        \coordinate (D) at (2,0);
        \coordinate (X) at (-1,0);
        \node [left] at (X) {$X$};
       
        \coordinate  [label=above right : $P {=} \pi_X(p_s)$] (P) at ($(C')!1!60:(D)$);
        \coordinate  [label=below right : $\bar{P}$] (bP) at ($(C')!1!-60:(D)$);
        \path[name path=line1] let \p1=($(P)-(X)$), \n1={atan2(\y1,\x1)} in (X)--($(X)+(180+\n1 : 2*2cm)$);
        \path[name path=line2] let \p2=($(bP)-(X)$), \n2={atan2(\y2,\x2)} in (X)--($(X)+(180+\n2 : 2*2cm)$);
        \path[name path=line0] (X)--(P);
        \path[name path=line3] (X)--(bP);

    \draw[name path = circle2](C') circle(2cm);
    \draw[name path =circle1, dotted](X) circle (0.5cm);
            \draw[name intersections={of=line1 and circle2,by={-P}}] (-P) node[ left]{$-P$}--(X);
        \draw[name intersections={of=line2 and circle2,by={-bP}}] (-bP) node[ left]{$-\bar{P}$}--(X);
        \path[name path=line3] (X)--(bP);
        \path[name path=linec] (X)--(C');
        
    \draw (X)--(P);
    \draw (X)--(C');
    \draw (X)--(bP);
    \fill[name intersections={of=line0 and circle1, by={Ps}}] {(Ps) circle (1pt) node[above]{$p_s$}};
     \fill[name intersections={of=line2 and circle1, by={-bPs}}] {(-bPs) circle (1pt) node[above]{$-\bar{p_s}$}};
      \fill[name intersections={of=line1 and circle1, by={-Ps}}] {(-Ps) circle (1pt) node[below]{$-p_s$}};
      \fill[name intersections={of=line3 and circle1, by={bPs}}] {(bPs) circle (1pt) node[below]{$\bar{p_s}$}};
      \fill[name intersections={of=linec and circle1, by={Cs'}}] {(Cs') circle (1pt) node[below right]{$c_s'$}};
    
    \end{tikzpicture}
    \caption{Description of $P,\bar{P},-P$, and $-\bar{P}$. The bigger circle represents $\partial \mathbf{B}_2'$ and the dotted circle represents
$\mathbb{S}^{m-1}\subset \mathbb{R}^m$ identified by $\mathbb{R}^m \cong T_XM$ via given parametrization of $M$}
    \label{points}
\end{figure}
\subsubsection{Properties of angles and distances} \label{section3.2.1}
In this section, we prove the lemmas which are essential in the proof of the main proposition in the next section. We prove a lemma about the ``symmetric properties" of angles and distances. In addition, we obtain a lemma which is motivated from the concept of solid angle. As a corollary, we introduce Newton's shell theorem with an infinitesimally thin ``shell" in ROSS. We begin with a lemma, which are useful for the lemmas below.

\begin{lemma}\label{rightangle}
    A triangle $(PQR)$ in ROSS $M$ with the metric $\textnormal{(\ref{metric})}$ satisfies :
     \begin{enumerate}[label=(\alph*)]
     \item If $M=\mathbb{S}^m$, $0<p,q,r<\pi$, and $p\le q <\frac{\pi}{2}$, then $\lambda(P)<\frac{\pi}{2}$.
     \item If $M=\mathbb{R}P^n, \mathbb{C}P^n, \mathbb{H}P^n,$ or $\mathbb{O}P^2$, $0<p,q,r<\frac{\pi}{2}$, and $p\le q <\frac{\pi}{4}$, then $\lambda(P)<\frac{\pi}{2}$.
     \item If $M$ is nCROSS, $0<p,q,r$, and $p\le q$, then $\lambda(P)<\frac{\pi}{2}$.
     \end{enumerate}
\end{lemma}
\begin{proof}
    \begin{enumerate}[label=(\alph*)]
        \item Suppose $\lambda(P)\ge \frac{\pi}{2}$. Using the law of cosines of spherical triangles (see p. 179 in  \cite{karcher1989comparison}),
        \begin{align*}
            \cos{p} = \cos{q}\cos{r}+\sin{q}\sin{r}\cos{P} < \cos{q}\cos{r}.
        \end{align*}
         Combining the previous inequality with $\cos{p},\cos{q}>0$, we obtain $\cos{r}>0$ and $\cos{p}<\cos{q}$. It implies $p>q$, contradiction to our assumption.
         
         \item Suppose $\lambda(P)\ge \frac{\pi}{2}$. Since $M$ has sectional curvature $K_M \le 4$ as in (\ref{curvature}), we can apply the triangle comparison theorem (see p. 197 in \cite{karcher1989comparison}). 
         \begin{align*}
             \cos{2p} \le \cos{2q}\cos{2r}+\sin{2q}\sin{2r}\cos{P}<\cos{2q}\cos{2r}.
         \end{align*}
         Then by an analogous argument in (a), we obtain a contradiction. 
         
         \item Suppose $\lambda(P)\ge \frac{\pi}{2}$. Since $M$ has sectional curvature $K_M \le -1$ as in (\ref{curvature}), we can apply the triangle comparison theorem (see p. 197 in \cite{karcher1989comparison}). 
         \begin{align*}
             \cosh{p} \ge \cosh{q}\cosh{r}-\sinh{q}\sinh{r}\cos{P}>\cosh{q}.
         \end{align*}
         Thus $p>q$, which contradicts to our assumption.
    \end{enumerate} \qed
\end{proof}

For $P\in \partial \mathbf{B}_2^\prime$, consider a triangle $(PXC')$ in $\textnormal{cl}(\mathbf{B}_2^\prime)$ defined in the beginning of Section 3, which consists of $X\in \mathbf{B}_2'$, the center $C'$ of $\mathbf{B}_2^\prime$, $P$, and geodesics connecting two of them. Then the next lemma explains relations of distances from X to $P,\bar{P},-P$, and $-\bar{P}$ and relations of angles at those points. 

\begin{figure}[!htb]
\minipage{0.33\textwidth}
  \includegraphics[width=\linewidth]{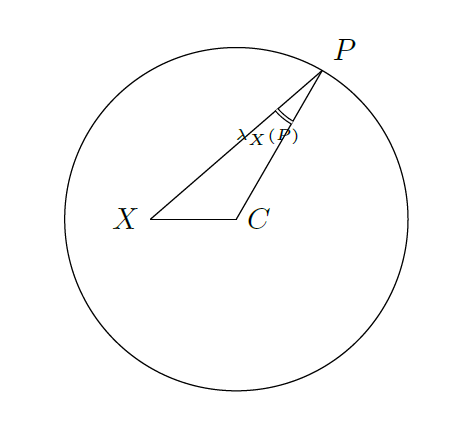}
\endminipage\hfill
\minipage{0.33\textwidth}
  \includegraphics[width=\linewidth]{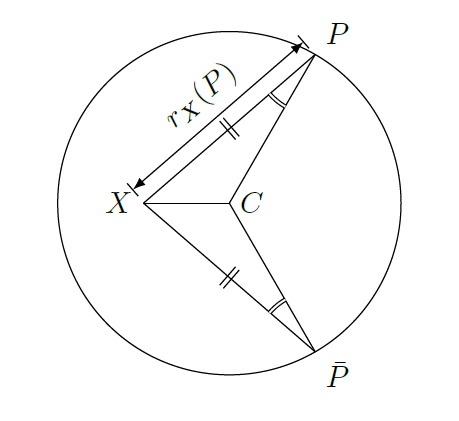}
\endminipage\hfill
\minipage{0.33\textwidth}%
  \includegraphics[width=\linewidth]{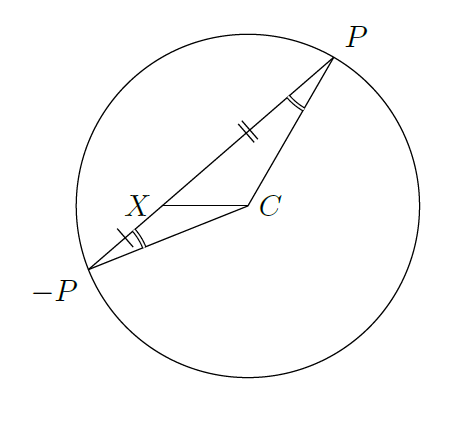}
\endminipage
\caption{Illustration of Lemma \ref{symmetry}. The circles represent $\partial \mathbf{B}_2'$}
\label{Lemma explained}
\end{figure}

\begin{lemma}\label{symmetry}
    Let $\lambda_X : \partial \mathbf{B}_2^\prime \rightarrow [0,\pi]$ be an angle function with respect to $X$ that assigns to each $P \in \partial \mathbf{B}_2^\prime$ an angle $\lambda(P)$ of the triangle (PXC'). Define $r_X$ as in the Proposition \ref{firsteigenfunction}. Then, $\lambda_X$ and $r_X$ satisfy the following. 
    \begin{enumerate}[label=(\alph*)]
        \item $0 \le \lambda_X(P) <\frac{\pi}{2}$.
        \item  $\lambda_X(P) = \lambda_X(\bar{P})$, $r_X(P)=r_X(\bar{P})$ for all $P\in \partial \mathbf{B}_2^\prime$. 
        \item $\lambda_X(P) = \lambda_X(-P)$, $r_X(P)\ge r_X(-P)$ for all $P\in \partial \mathbf{B}_2^\prime$ satisfying\\ $\angle(\vec{v}_X(P), \vec{v}_X(C')) \le \frac{\pi}{2}$. The equality holds if and only if $\angle(\vec{v}_X(P), \vec{v}_X(C')) = \frac{\pi}{2}$. 
    \end{enumerate}

\end{lemma}

\begin{proof} We will prove this lemma when $M = \mathbb{C}P^n, \mathbb{H}P^n$ or $\mathbb{O}P^2$. Then we have $R_2 < \frac{inj (M)}{2}=\frac{\pi}{4}$.
    
    \begin{enumerate}[label=(\alph*)]
    \item Note that $R_2 <\frac{\pi}{4}$ and $|C'X|<|C'P|=R_2$. Then the statement follows from Lemma \ref{rightangle}.
    
    \item Consider two triangles $(PXC')$ and $(\bar{P}XC')$. By the constructions of $P$ and $\bar{P}$, $\lambda(X)$ of $(PXC')$ and $(\bar{P}XC')$ are identical. The same holds for $\varphi(X)$. Note that the two triangles have the common edge $XC'$ and $|C'P|=|C'\bar{P}|=R_2$. From the fact that $|C'X|<|C'P|=R_2<\frac{\pi}{4}$ we have $\sin{|C'X|}<\cos{|C'P|}$. Therefore by Proposition \ref{congruence}, $(PXC')$ and $(\bar{P}XC')$ are congruent. Then our statement follows.
    
    \item Using the fact that $\mathbf{B}_2^\prime$ is convex (see p. 148 in \cite{petersen2006riemannian}), we can define a point $R \in \mathbf{B}_2^\prime$ in the complete geodesic containing $X$ and $P$ such that the geodesic meets $C'R$ perpendicularly. We claim that $\lambda(\vv{XR},\vv{XC'})\le \frac{\pi}{2}$. If $X=R$, $\lambda(\vv{XR},\vv{XC'})= \frac{\pi}{2}$. Otherwise, we have $|RC'|<|XC'|<\frac{\pi}{4}$. Then by Lemma \ref{rightangle} for $(XC'R)$, our claim follows. Then the condition on $P$ in the statement implies $R\in XP$, so $|PR|\le |PX|$. On the other hand, two triangles $(PRC')$  and $(-PRC')$ are congruent by (\ref{180}),(\ref{180_2}), and Proposition \ref{congruence} as in the proof of (b). Thus we obtain that $\lambda_X(P)=\lambda_X(-P)$ and $|PR|=|-PR|$, which imply the desired conclusion. 
     
    \end{enumerate}
     A slight change in the proof shows it also holds if $M$ is $\mathbb{S}^m, \mathbb{R}P^n$ or nCROSSs. \qed
\end{proof}

Now we will give another lemma that explains an ``infinitesimal area of $\partial \mathbf{B}_2$ from $X$" can be calculated by $\lambda_X$ and $r_X$.  

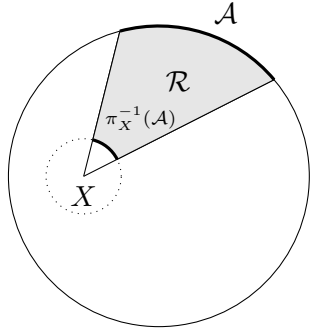
\begin{figure}[htb!]
    \hspace{1cm}
    \begin{tikzpicture}
     \coordinate (C) at (0,0);
    \coordinate[label =below : $X$] (X) at (-1,0);
    \coordinate (D) at (2,0);
    \coordinate (P1) at ($(C)!1!105:(D)$);
    \coordinate (P2) at ($(C)!1!40:(D)$);
    \path[name path=line2] let \p2=($(P1)-(X)$), \n2={atan2(\y2,\x2)} in (X) --($(X)+(180+\n2 :2*2cm)$);
    \path[name path=line3] let \p3=($(P2)-(X)$), \n3={atan2(\y3,\x3)} in (X) --($(X)+(180+\n3:2*2cm)$);
    \draw[name path = circle2](C) circle(2cm);
    \draw[name path =circle1, dotted](X) circle (0.5cm);
    \draw (P1)--(X);
    \draw (P2)--(X);
    \draw[fill = gray!20] let \p2=($(P1)-(X)$), \n2={atan2(\y2,\x2)},\p3=($(P2)-(X)$), \n3={atan2(\y3,\x3)} in (X) + (\n3 : 0.5) arc (\n3 : \n2 : 0.5)
    -- (P1) arc (105 :40 : 2)--cycle;
    \node[above right] at ($(C)!1!72.5:(D)$) {$\mathcal{A}$};
    \node[above right, font=\tiny] at ($(X)!1!72.5:(-0.5,0)$) {$\pi_X^{-1}(\mathcal{A})$};
    \draw[very thick] let \p2=($(P1)-(X)$), \n2={atan2(\y2,\x2)},\p3=($(P2)-(X)$), \n3={atan2(\y3,\x3)} in (X) + (\n3 : 0.5) arc (\n3 : \n2 : 0.5);
    \draw[very thick]  (P1) arc (105 : 40 : 2);
    \node[above right] at ($(C)!1!92.5:(1,0)$) {$\mathcal{R}$};
    \end{tikzpicture}
    \vspace*{-17mm}
    \caption{Description of $
    \mathcal{R}$ in the proof of Lemma \ref{radon}. The dotted circle and the bigger circle represent $\partial\mathbf{B}_1$ and $\partial \mathbf{B}_2'$, respectively.}
\end{figure}

\begin{lemma} \label{radon}
    Let $\mu$ be the Lebesgue measure on $\mathbb{S}^{m-1}$ and consider the pushforward ${\pi_X}_{\#}\mu$ on $\partial \mathbf{B}_2^\prime$. Then for a measurable set $\mathcal{A} \subset \partial \mathbf{B}_2^\prime$, we have
    \begin{align*}
        {\pi_X}_{\#}\mu(\mathcal{A})= \mu(\pi_X^{-1}(\mathcal{A}))= \int_{\mathcal{A}} \frac{\cos{\lambda_X}}{\omega(r_X)}dS_2^\prime,
    \end{align*}
    where $S_2^\prime$ is the induced measure on $\partial \mathbf{B}_2^\prime$ from the metric of $M$. Equivalently, 
    \begin{align*}
        dS_2^\prime = \frac{\omega(r_X)}{\cos{\lambda_X}}d{\pi_X}_{\#}\mu.
    \end{align*}
\end{lemma}

\begin{proof}
    It is clear that $S_2^\prime$ and ${\pi_X}_{\#}\mu$ are $\sigma$-finite and  ${\pi_X}_{\#}\mu \ll S_2^\prime$ that is to say $(\pi_X)_{\#}\mu$ is absolutely continuous with respect to $S_2^\prime$. Furthermore, $S_2^\prime \ll {\pi_X}_{\#}\mu$. By the Radon-Nikodym theorem, there are functions $f_1$ and $f_2$ on $\partial \mathbf{B}_2^\prime$ such that 
    \begin{align*}
    {\pi_X}_{\#} \mu(\mathcal{A})=\int_{\mathcal{A}}f_1 dS_2^\prime
    \end{align*}
    and 
    \begin{align*}
    dS_2^\prime(\mathcal{A})=\int_{\mathcal{A}}f_2 d{\pi_X}_{\#}\mu.
    \end{align*}
  Consider a vector field $\mathbb{F}$ on $M\setminus \{X\}$ defined by
     \begin{align*}
        \mathbb{F}(Y) = \left(\frac{1}{\omega(r_X)}\frac{\partial}{\partial r}\right)(Y),
    \end{align*}
    where $\frac{\partial}{\partial r}(Y)$ is the vector in $T_YM$ obtained by the parallel transport of the unit tangent vector $\vec{v}_X(Y)$ along $XY$.
    Then 
    \begin{align}\label{div}
        \text{div}(\mathbb{F})= \frac{1}{\omega(r_X)}\frac{\partial}{\partial r}\left(\omega(r_X)\cdot \frac{1}{\omega(r_X)} \right)=0.
    \end{align}
 Let $\mathcal{B}$ be a geodesic ball in $\partial \mathbf{B}_2'$ with respect to the induced metric on $\partial \mathbf{B}_2'$. We now consider a region $\mathcal{R}$ that is the region in $\text{cl}(\mathbf{B}_2')\setminus \mathbf{B}_1$ of the solid cone from $X$ over $\mathcal{B}$. Equivalently, 
    \begin{align*}
        \mathcal{R} = \{\text{exp}_X (t\cdot \vec{v}_X(Y)) | Y\in \mathcal{B}, R_1 \le t \le r_X(Y)\}.
    \end{align*}
   Let $\mathcal{B}_1=\mathcal{R}\cap \partial \mathbf{B}_1 $. Then applying the divergence theorem to $\mathbb{F}$ on $\mathcal{R}$, we have 
   \begin{align*}
       0=\int_{\mathcal{R}}\text{div}\mathbb{F} = \int_{\mathcal{B}} \frac{\cos{\lambda_X}}{\omega(r_X)}dS_2^\prime-\int_{\mathcal{B}_1} \frac{1}{\omega(R_1)}dS_1,
   \end{align*}
   where $S_1$ is the measure on $\partial \mathbf{B}_1$ induced by the metric of $M$. 
   Combining it with the fact that 
   \begin{align*}
   \int_{\mathcal{B}_1} \frac{1}{\omega(R_1)} dS_1= \mu(\pi_X^{-1}(\mathcal{A})),
   \end{align*}
   the first statement is proved for $\mathcal{B}$. Then by Theorem 4.7 in \cite{Simon1983measure}, the first statement is proved. Since $\cos{\lambda_X}\neq 0$ from Lemma \ref{symmetry}, the second argument follows.  \qed
\end{proof}

\begin{remark}
If we extend the domain of $a(r)$ to $(0,\infty)$ in Lemma \ref{a(r)}, the vector field $\mathbb{F}$  in the proof is, in fact, $\mathbb{F}(Y)= \nabla(a\circ r_X)(Y)$. Note that $a\circ r_X$ is harmonic in $M\setminus \{X\}$. Thus, (\ref{div}) is obtained without computation.

\end{remark} 

The following corollary is not necessary for the proof of the main theorem. 
\begin{corollary} \label{newton} We have
\begin{align*}
    \int_{\partial \mathbf{B}_2^\prime} \frac{\vec{v}_X}{\omega(r_X)}dS_2^\prime =0.
\end{align*}    
\end{corollary}
\begin{proof}
    Using the previous lemma, the left hand side is equal to 
    \begin{align} \label{symmetric identity}
        \int_{\mathbb{S}^{m-1}} \left(\frac{\vec{v}_X}{\cos{\lambda_X}}\circ \pi_X\right) d\mu.
    \end{align}
    By Lemma \ref{symmetry}, we have 
    \begin{align} \label{identityinNewtonshelltheorem}
        \left(\frac{\vec{v}_X}{\cos{\lambda_X}} \right) \circ \pi_X(p_s) + \left(\frac{\vec{v}_X}{\cos{\lambda_X}} \right) \circ \pi_X(-p_s) =0
    \end{align} for $p_s \in \mathbb{S}^{m-1}$ (see Fig. \ref{integration at X inNewtonshelltheorem}). Then this relation gives the desired result. 
    
    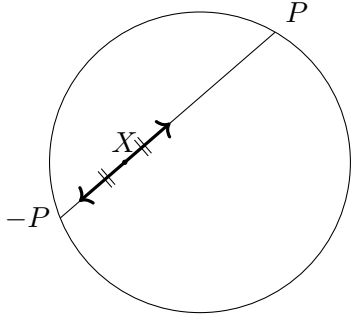
\begin{figure}
   
    \begin{center}
     \hspace{-1.8cm}
         \begin{tikzpicture}
    \coordinate (C) at (0,0);
        \coordinate (D) at (2,0);
        \coordinate (X) at (-1,0);
    \fill {(X) circle (1pt) node [above]{$X$}};
        \coordinate  [label=above right : $P$] (P) at ($(C)!1!60:(D)$);
        \path[name path=line1] let \p1=($(P)-(X)$), \n1={atan2(\y1,\x1)} in (X)--($(X)+(180+\n1 : 2*2cm)$);
        \path[name path=line0] (X)--(P);
    \draw[name path = circle2](C) circle(2cm);
    \draw[name intersections={of=line1 and circle2,by={-P}}] (-P) node[ left]{$-P$}--(X);
    \draw (X)--(P);
    \coordinate (P_s) at ($(X)!0.8cm!(P)$);
    \draw [->, very thick] (X) -- (P_s);
    \coordinate (-P_s) at ($(X)!0.8cm!(-P)$);
    \draw [->, very thick] (X) -- (-P_s);
    \tkzMarkSegment[pos=.4,mark=||](X,P_s);
    \tkzMarkSegment[pos=.4,mark=||](X,-P_s);
    \end{tikzpicture}
    \end{center}
        \caption{Pictorial explanation of calculation of (\ref{identityinNewtonshelltheorem}). Each thick arrows represents integrand of (\ref{identityinNewtonshelltheorem}) at $p_s$ and $-p_s$}.
        \label{integration at X inNewtonshelltheorem}
    \end{figure}
    \qed
\end{proof}

\begin{remark}
    Note that if $M=\mathbb{R}^{3}$, then $\omega(r)=r^2$. Furthermore $\vec{v}_X(\pi_X(p_s))$ is the unit vector from $X$ to $P=\pi_X(p_s)$ at $X$. Thus the equation becomes Newton's shell theorem, which implies that the net gravitational force of a spherical shell acting on any object inside is zero. 
\end{remark}

\subsubsection{The proof for nCROSS} \label{section3.2.2}
In this section, we prove the main theorem for nCROSS. We use the fact that the first mixed Steklov-Dirichlet eigenfunction, $a\circ r_C$, of the annulus $\mathbf{B}_2\setminus \mathbf{B}_1$ is a test function in both of the variational characterizations of $\sigma_1(\mathbf{B}_2^\prime \setminus \mathbf{B}_1)$ and $\sigma_1(\mathbf{B}_2\setminus \mathbf{B}_1)$. Substituting the test function into the two Rayleigh quotients, we compare the two denominators and the two numerators in the following two propositions.\\
Define a map
\begin{align*}
    \int_{\partial\mathbf{B}_2^\prime}(a\circ r_{(\cdot)})^2 dS_2^\prime : \mathbf{B}_2^\prime &\rightarrow \mathbb{R}
\end{align*} 
that assigns to $X \in \mathbf{B}_2^\prime$
\begin{align*}
    \int_{\partial\mathbf{B}_2^\prime}(a\circ r_{X})^2 dS_2^\prime
\end{align*}
In the following proposition, we show that the function has a minimum value at $C$ by analyzing the gradient of the function at each $X\in \mathbf{B}_2^\prime$,
\begin{align*}
     \nabla \left(\int_{\partial \mathbf{B}_2^\prime} (a\circ r_{(\cdot)})^2 dS_2^\prime \right)(X)\in T_XM.
\end{align*}
\begin{proposition}\label{gradientofL2sum}
We have
\begin{align} \label{gradientofthefunction}
    \nabla \left(\int_{\partial \mathbf{B}_2^\prime} (a\circ r_{(\cdot)})^2 dS_2^\prime \right)(X) =
    \begin{cases}
    -g(X)\cdot \vec{v}_X(C') &\text{if}\enspace X\neq C',\\
    0 &\text{if}\enspace X=C',
    \end{cases}
\end{align}
where $g: \mathbf{B}_2^\prime \setminus\{C'\} \rightarrow \mathbb{R}^+$ is a positive function. Furthermore, 
\begin{align}\label{inequalityfromgradient} 
\int_{\partial \mathbf{B}_2^\prime} (a\circ r_{C'})^2 dS_2^\prime  \le \int_{\partial \mathbf{B}_2^\prime} (a\circ r_X)^2 dS_2^\prime,
\end{align}
and equality holds if and only if $X=C'$.
\end{proposition}

\begin{proof}The gradient is calculated at $X\in \mathbf{B}_2^\prime$, so it does not affect on the integration region $\partial \mathbf{B}_2^\prime$. Then for $P\in \partial \mathbf{B}_2^\prime$, $\nabla (a\circ r_{(\cdot)}(P))^2(X) \in T_X M$. Thus
\[ -\nabla \left( \int_{\partial \mathbf{B}_2^\prime} (a\circ r_{(\cdot)})^2(P) dS_2^\prime(P)\right)(X)
=\int_{\partial \mathbf{B}_2^\prime}\frac{2(a\circ r_X)}{\omega(r_X)}\cdot(-\nabla r_{(\cdot)}(P)(X))dS_2^\prime(P).
\]
With $-\nabla (r_{(\cdot)}(P))(X)=\vec{v}_X(P)$ and Lemma \ref{radon}, the previous equation is equal to
\begin{align} \label{symmetric identity2}
 \int_{\mathbb{S}^{m-1}} \left( 2(a\circ r_X) \cdot \frac{\vec{v}_X}{\cos{\lambda_X}} \right) \circ \pi_X d\mu.
\end{align}

\begin{figure}
	\centering
    \hspace{-1.5cm}
    \begin{tikzpicture}
        \coordinate (C) at (0,0);
        \coordinate (D) at (2,0);
        \coordinate (X) at (-1,0);
        \fill {(X) circle (1pt)};
        \node[left] at (-1.2,0) {$X$};
        \coordinate  [label=above right : $P$] (P) at ($(C)!1!60:(D)$);
        \coordinate  [label=below right : $\bar{P}$] (bP) at ($(C)!1!-60:(D)$);
        \path[name path=line1] let \p1=($(P)-(X)$), \n1={atan2(\y1,\x1)} in (X)--($(X)+(180+\n1 : 2*2cm)$);
        \path[name path=line2] let \p2=($(bP)-(X)$), \n2={atan2(\y2,\x2)} in (X)--($(X)+(180+\n2 : 2*2cm)$);
        \path[name path=line0] (X)--(P);
        \path[name path=line3] (X)--(bP);
        \draw[name path = circle2](C) circle(2cm);
        \draw[name intersections={of=line1 and circle2,by={-P}}] (-P) node[ left]{$-P$}--(X);
        \draw[name intersections={of=line2 and circle2,by={-bP}}] (-bP) node[ left]{$-\bar{P}$}--(X);
        \path[name path=line3] (X)--(bP);
        \path[name path=linec] (X)--(C);
        \draw (X)--(P);
        \draw (X)--(bP);
        \coordinate (P_s) at ($(X)!0.8cm!(P)$);
        \draw [->, very thick] (X) -- (P_s);
        \coordinate (bP_s) at ($(X)!0.8cm!(bP)$);
        \draw [->, very thick] (X) -- (bP_s);
        \tkzMarkSegment[pos=.4,mark=||](X,P_s);
        \tkzMarkSegment[pos=.4,mark=||](X,bP_s);
        \coordinate (-P_s) at ($(X)!0.4cm!(-P)$);
        \coordinate (-bP_s) at ($(X)!0.4cm!(-bP)$);
        \draw [->, very thick] (X) -- (-P_s);
        \draw [->, very thick] (X) -- (-bP_s);
        \tkzMarkSegment[pos=.4,mark=|](X,-P_s);
        \tkzMarkSegment[pos=.4,mark=|](X,-bP_s);
    \end{tikzpicture}
    \caption{Pictorial explanation of calculation of (\ref{symmetric identity2}). Each thick arrows represents integrand of (\ref{symmetric identity2}) at $p_s,\bar{p_s},-p_s$, and $-\bar{p_s}$.}
\label{integration at X}
\end{figure}
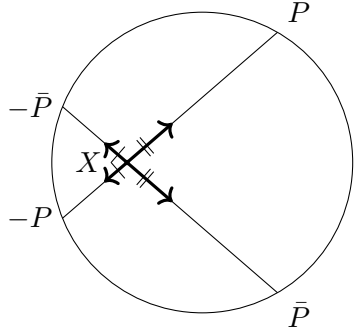
If $X=C'$, the integral has value 0. Otherwise, we consider the integrand at $p_s \in \{v|\langle v, c' \rangle \ge 0 \}\subset \mathbb{S}^{m-1}$, $\bar{p_s}, -p_s,$ and $-\bar{p_s}$ (see Fig.  \ref{integration at X}). Note that the condition for {$p_s$} is equivalent to $\angle(\vec{v}_X(P),\vec{v}_X(C'))\le \frac{\pi}{2}$. Then using Lemma \ref{symmetry},

\begin{align*}
     &\left(\left( 2(a\circ r_X) \cdot \frac{\vec{v}_X}{\cos{\lambda_X}} \right)(P)+\left( 2(a\circ r_X) \cdot \frac{\vec{v}_X}{\cos{\lambda_X}} \right)(\bar{P})\right) \\
     +&\left(\left( 2(a\circ r_X) \cdot \frac{\vec{v}_X}{\cos{\lambda_X}} \right)(-P)+\left( 2(a\circ r_X) \cdot \frac{\vec{v}_X}{\cos{\lambda_X}} \right)(-\bar{P})\right) \\
     =&2(a\circ r_X)(P)\cdot \frac{2\langle\vec{v}_X(P),\vec{v}_X(C')\rangle }{\cos{\lambda_X}}\cdot \vec{v}_X(C')\\
     +&2(a\circ r_X)(-P)\cdot \frac{2\langle\vec{v}_X(-P),\vec{v}_X(C')\rangle }{\cos{\lambda_X}}\cdot \vec{v}_X(C')\\
     =&4\left( (a\circ r_X)(P)-(a\circ r_X)(-P)\right)\cdot \frac{\langle\vec{v}_X(P),\vec{v}_X(C')\rangle }{\cos{\lambda_X}}\cdot \vec{v}_X(C').
\end{align*}
Furthermore, Lemma \ref{symmetry} implies $ (a\circ r_X)(P)-(a\circ r_X)(-P)>0$ unless $\angle(\vec{v}_X(P),\vec{v}_X(C')) =\pi/2$. Thus our integration has a form $g(X)\cdot \vec{v}_X(C')$ for some positive function $g$. Note that we actually proved that the gradient of the function has the opposite direction from $X$ to $C'$ (see Fig. \ref{pic:gradientofthefunction}). It implies our desired inequality (\ref{inequalityfromgradient}).

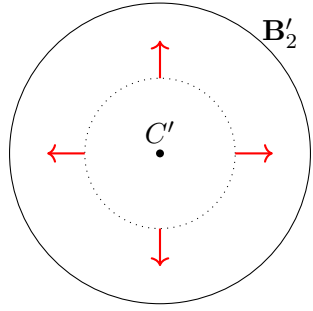
\begin{figure}
        \begin{center}
    \begin{tikzpicture}
    \coordinate [label=above : $C'$] (C) at (0,0);
\coordinate (D) at (2,0);
\coordinate (X) at (-1,0);
\coordinate (rX) at (1,0);
\coordinate (nX) at (0,1);
\coordinate (sX) at (0,-1);

\draw(C) circle (2cm);
\draw[dotted] (C) circle (1cm);
\fill [black] (C) circle (1.5pt);

\coordinate (P_s) at (-1.5,0);
\draw [->, thick, red] (X) -- (P_s);
\coordinate (rP_s) at (1.5,0);
\draw [->, thick, red] (rX) -- (rP_s);
\coordinate (nP_s) at (0,1.5);
\draw [->, thick, red] (nX) -- (nP_s);
\coordinate (sP_s) at (0,-1.5);
\draw [->, thick, red] (sX) -- (sP_s);

\node at (1.6,1.6){$\mathbf{B}_2^\prime$};    
    \end{tikzpicture}
    \end{center}
    \caption{Pictorial description of (\ref{gradientofthefunction}).}
    \label{pic:gradientofthefunction}
\end{figure} \qed
\end{proof}

\begin{remark}
    \begin{itemize}
        \item In the proof, the function $g$ only depends on the distance between $X$ and $C'$.
        \item The proof is similar to the proof of Corollary \ref{newton} if we compare (\ref{symmetric identity}) and (\ref{symmetric identity2}). The difference between the two proofs is the fact that $a$ is an increasing function.
    \end{itemize}
\end{remark}

\begin{corollary}\label{L2sum}We have
    \[\int_{\partial \mathbf{B}_2} (a\circ r_C)^2 dS_2 \le \int_{\partial \mathbf{B}_2^\prime} (a\circ r_C)^2 dS_2^\prime,
    \]
    where $S_2$ is the measure on $\partial \mathbf{B}_2$ induced from the metric of $M$. The equality holds if and only if $\mathbf{B}_2^\prime = \mathbf{B}_2$.
\end{corollary}
\begin{proof}
    Note that $\mathbf{B}_2$ is a ball of radius $R_2$, centered at $C$. Therefore we have 
    \[\int_{\partial \mathbf{B}_2} (a\circ r_C)^2 dS_2 = \int_{\partial \mathbf{B}_2^\prime} (a\circ r_{C'})^2 dS_2^\prime.
    \]
    Then Proposition \ref{gradientofL2sum} implies the statement. \qed
\end{proof}
In the following proposition, $(\nabla (a\circ r_X))(Z)$ for $Z\in M \setminus \{X\}$ is the gradient of 
\begin{align*}
    a\circ r_X(\cdot) : M\setminus\{X\} \rightarrow \mathbb{R}
\end{align*} at $Z$. 

\begin{proposition}\label{gradient} We have
\[ \int_{\mathbf{B}_2^\prime \setminus \textnormal{cl}(\mathbf{B}_1)}|\nabla(a\circ r_C)|^2dV \le \int_{\mathbf{B}_2 \setminus \textnormal{cl}(\mathbf{B}_1)}|\nabla(a\circ r_C)|^2 dV,
\] where $V$ is the induced measure of $M$. The equality holds if and only if $\mathbf{B}_2^\prime = \mathbf{B}_2$.
\end{proposition}
\begin{proof}
    Note that $|\nabla(a\circ r_C(\cdot))|=|\nabla a|\circ r_C(\cdot)$ and it is easy to check that $|\nabla a|(r) = |a^\prime(r)| = \frac{1}{\omega(r)}$ is a decreasing function since we only consider when $M$ is nCROSS. Then
    \begin{align*} &\int_{\mathbf{B}_2 \setminus \textnormal{cl}(\mathbf{B}_1)} |\nabla (a\circ r_C)|^2 dV- \int_{\mathbf{B}_2^\prime \setminus \textnormal{cl}(\mathbf{B}_1)} |\nabla (a\circ r_C)|^2 dV \\
    =&\int_{\mathbf{B}_2 \setminus \mathbf{B}_2^\prime} |\nabla (a\circ r_C)|^2 dV - \int_{\mathbf{B}_2^\prime \setminus \mathbf{B}_2} |\nabla (a\circ r_C)|^2 dV \\
    \ge &\int_{\mathbf{B}_2 \setminus \mathbf{B}_2^\prime} | \nabla a(R_2)|^2 dV - \int_{\mathbf{B}_2^\prime \setminus \mathbf{B}_2} |\nabla a(R_2)|^2 dV =0.
    \end{align*}
    To satisfy the equality, $|\mathbf{B}_2^\prime \setminus \mathbf{B}_2|=|\mathbf{B}_2 \setminus \mathbf{B}_2^\prime|=0$, or $\mathbf{B}_2^\prime=\mathbf{B}_2$. \qed
\end{proof}
\begin{remark}
  We used only the fact that $\omega(r)$ is a concave function in $[0,2R_2)$. Thus the proof also applies when $M$ is CROSS and $R_2 < \frac{inj(M)}{4}$.  
\end{remark}
 Now we have the following proof of the main theorem when $M$ is a nCROSS.
\paragraph{Proof of Theorem \ref{main} for nCROSS}
    Note that $u\circ r_C =0$ on $\partial \mathbf{B}_1$. By the variational characterization of $\sigma_1(\mathbf{B}_2^\prime \setminus \textnormal{cl}(\mathbf{B}_1))$, 
    \[ \sigma_1(\mathbf{B}_2^\prime \setminus\textnormal{cl}(\mathbf{B}_1)) \le \frac{\int_{\mathbf{B}_2^\prime\setminus \textnormal{cl}(\mathbf{B}_1)} |\nabla (a\circ r_C)|^2 dV}{\int_{\partial \mathbf{B}_2^\prime}(a\circ r_C)^2 dS_2^\prime}.
    \]
    By Corollary \ref{L2sum} and Proposition \ref{gradient}, we have
    \[ \sigma_1(\mathbf{B}_2^\prime \setminus\textnormal{cl}(\mathbf{B}_1)) \le \frac{\int_{\mathbf{B}_2 \setminus \textnormal{cl}(\mathbf{B}_1)} |\nabla (a\circ r_C)|^2 dV}{\int_{\partial \mathbf{B}_2}(a\circ r_C)^2 dS_2}.
    \]
    Since we have shown that $a\circ r_C$ is the first mixed Steklov-Dirichlet eigenfunction on the annulus $\mathbf{B}_2 \setminus \textnormal{cl}(\mathbf{B}_1)$ in Proposition \ref{firsteigenfunction}, the right hand side is $\sigma_1(\mathbf{B}_2 \setminus \textnormal{cl}(\mathbf{B}_1))$. It is the desired inequality. In addition, the equality condition is followed from the equality conditions in  Corollary \ref{L2sum} and Proposition \ref{gradient}. \qed

\begin{remark}
    The method of the proof carries over to Euclidean space $\mathbb{R}^m$.
\end{remark}
\subsubsection{The proof for CROSS} \label{section3.2.3}
In this section we modify the proof of Proposition \ref{gradient} to show that the inequality in this proposition also holds when $M$ is CROSS and $R_2<\frac{inj(M)}{2}$. Then using the same argument in Section \ref{section3.2.2}, we can show that the main theorem holds in this situation.

$B_r(C)$ denotes the ball of radius $r$, centered at $C$ and $d:=r_C(C')$ denotes the distance between $C$ and $C'$. In addition, let $\Pr_C : M\setminus\{C\}\rightarrow \mathbb{S}^{m-1}\subset \mathbb{R}^m$ be the direction of a point in $M\setminus \{C\}$ with respect to $C$ in the coordinate of $M$ we defined in Definition 4 in Section \ref{section3.2}.  Then the difference between the two sides of the inequality in Proposition \ref{gradient} becomes
\begin{align*}
    &\int_{\mathbf{B}_2\setminus \text{cl}(\mathbf{B}_1)}\left(\frac{1}{\omega(r)}\right)^2dV-\int_{\mathbf{B}_2'\setminus \text{cl}(\mathbf{B}_1)}\left(\frac{1}{\omega(r)}\right)^2dV \\
    =&\int_{\mathbf{B}_2\setminus \mathbf{B}_2'}\left(\frac{1}{\omega(r)}\right)^2dV-\int_{\mathbf{B}_2'\setminus \mathbf{B}_2}\left(\frac{1}{\omega(r)}\right)^2dV \\
    =&\int_{R_2-d}^{R_2}\int_{\Pr_C((\mathbf{B}_2\setminus \mathbf{B}_2')\cap \partial B_{r_1}(C))}\frac{1}{\omega(r_1)}d\mu dr_1\\ 
    &-\int_{R_2}^{R_2+d}\int_{\Pr_C((\mathbf{B}_2'\setminus \mathbf{B}_2)\cap \partial B_{r_2}(X))}\frac{1}{\omega(r_2)}d\mu dr_2\\
    =&\int_0^d \left( \int_{\Pr_C((\mathbf{B}_2\setminus \mathbf{B}_2')\cap \partial B_{R_2-s}(C))}\frac{1}{\omega(R_2-s)} d\mu \right. \\
    &\left. -\int_{\Pr_C((\mathbf{B}_2'\setminus \mathbf{B}_2)\cap \partial B_{R_2+s}(C))}\frac{1}{\omega(R_2+s)} d\mu \right) ds.
\end{align*} 
The last equality is obtained by substituting $r_1$ and $r_2$ by $R_2-s$ and $R_2+s$ for $s<d$, respectively. See Fig. \ref{regions} for pictorial description of  $(\mathbf{B}_2\setminus \mathbf{B}_2')\cap \partial B_{R_2-s}(C)$ and $(\mathbf{B}_2'\setminus \mathbf{B}_2)\cap \partial B_{R_2+s}(C)$. Then the integral becomes nonnegative provided that the following two lemmas hold.

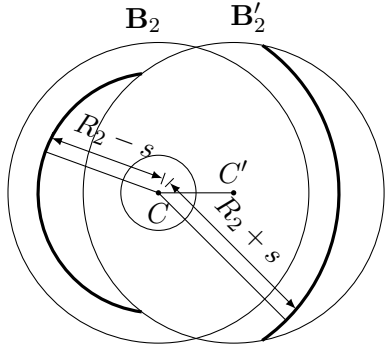
\begin{figure}
    \centering
    \begin{tikzpicture}
    \coordinate (C) at (0,0);
        \coordinate (D) at (2,0);
        \coordinate (X) at (-1,0);
        \fill {(X) circle (1pt) node[below]{$C$} };
        \fill {(C) circle (1pt) node[above]{$C'$}};
        \node[above] at (-1.2,2) {$\mathbf{B}_2$};
        \node[above] at (0.2,2) {$\mathbf{B}_2'$};
        \draw[name path = circle2](C) circle(2cm);
        \draw[name path = circle3](X) circle(2cm);
    \draw[name path =circle1](X) circle (0.5cm);
    \draw[style=transparent, name path=circle3'](X) circle (2.4cm);
    \draw[style=transparent, name path=circle3''](X) circle (1.6cm);
    \draw (X)--(C);
    
     \draw [name intersections={of=circle2 and circle3', by={first intersect, second intersect}}];
     \draw [name intersections={of=circle2 and circle3'', by={third intersect, fourth intersect}}];
     \draw[very thick] let \p1=($(first intersect)+(1,0)$), \n1={atan2(\y1,\x1)},\p2=($(second intersect)+(1,0)$), \n2={atan2(\y2,\x2)} in (X) + (\n2 : 2.4) arc (\n2 : \n1 :2.4);
     \draw[very thick] let \p3=($(third intersect)+(1,0)$), \n3={atan2(\y3,\x3)},\p4=($(fourth intersect)+(1,0)$), \n4={atan2(\y4,\x4)} in (X) + (\n4+180:-1.6) arc (\n4+180 : \n3-180 :-1.6);
    \draw (X)to [dim above = $R_2+s$] ($(X)!1!-45:(1.4,0)$);
    \draw (X)to [dim below = $R_2-s$] ($(X)!1!1600:(0.6,0)$);
    \end{tikzpicture}
    \caption{The left and right thick arcs represent $(\mathbf{B}_2\setminus \mathbf{B}_2')\cap \partial B_{R_2-s}(C)$ and  $(\mathbf{B}_2'\setminus \mathbf{B}_2)\cap \partial B_{R_2+s}(C)$, respectively. In addition, the inner circle is $\partial \mathbf{B}_1$ and we have $|CC'|=d$.}
\label{regions}
\end{figure}

\begin{lemma}\label{measure comparison}
We have
\begin{align*}
    |{\Pr}_{C}((\mathbf{B}_2'\setminus \mathbf{B}_2)\cap \partial B_{R_2+s}(C))|\le
    |{\Pr}_C((\mathbf{B}_2\setminus \mathbf{B}_2')\cap \partial B_{R_2-s}(C))|
\end{align*}
for $s<R_2$.
\end{lemma}
\begin{proof}
    Consider $S \in (\mathbf{B}_2\setminus \mathbf{B}_2')\cap \partial B_{R_2-s}(C)$. Then the triangle $(SCC')$ has side lengths
    \begin{align*}
        |CC'|=d, |CS|=R_2-s, |C'S|\ge R_2.
    \end{align*}
    Consider the space form $\mathbb{S}_{\kappa}^m$ of constant curvature $\kappa$, where $\kappa\in \mathbb{R}^+$ is a constant such that a geodesic ball of radius $R_2$ is a hemisphere in $\mathbb{S}_{\kappa}^m$. Then we have
    \begin{align*}
        \frac{\pi}{2\sqrt{\kappa}}=R_2,
    \end{align*}
    so $\kappa$ is bigger than the sectional curvature of $M$. Now consider a triangle $(S_{\kappa}C_{\kappa}C'_{\kappa})$ with the same side lengths as $(SCC')$ in $\mathbb{S}_{\kappa}^m$. Then by the triangle comparison theorem (see \cite[p. 197]{karcher1989comparison}), 
    \begin{align*}
        \angle{SCC'} \le \angle{S_{\kappa}C_{\kappa}C'_{\kappa}}.
    \end{align*} 
    Then it implies the following inequality.
    \begin{align} \label{measure1}
        &|{\Pr}_{C}((\mathbf{B}_2\setminus \mathbf{B}_2')\cap \partial B_{R_2-s}(C))|\nonumber\\
        =&|\{{\Pr_C}(S)| |CS|=R_2-s, |C'S|\ge R_2 \}|\nonumber\\
        \ge& |\{S_{\kappa}| |C_{\kappa}S_{\kappa}|=R_2-s, |C'_{\kappa}S_{\kappa}|\ge R_2\}|\times \frac{1}{s_{\kappa}(R_2-s)}\nonumber\\
        =&|\{S_{\kappa}| S_{\kappa} \in ((\mathbf{B}_2)_{\kappa}\setminus (\mathbf{B}_2')_{\kappa})\cap \partial (B_{R_2-s})_{\kappa}(C_{\kappa})\}|\times \frac{1}{s_{\kappa}(R_2-s)},
    \end{align}
    where $(\mathbf{B}_2)_{\kappa}$ and $(\mathbf{B}_2')_{\kappa}$ are geodesic balls of radius $R_2$ in $\mathbb{S}_{\kappa}^m$, centered at $X_{\kappa}$ and $C'_{\kappa}$, respectively, and
    \begin{align*}
       s_{\kappa}(r)=\frac{\sin{\sqrt{\kappa}r}}{\sqrt{\kappa}}.  
    \end{align*}
   By a similar argument, we obtain
   \begin{align}\label{measure2}
       &|{\Pr}_C((\mathbf{B}_2'\setminus \mathbf{B}_2)\cap \partial B_{R_2+s}(C))|\nonumber\\
       \le& |\{S_{\kappa}'| S_{\kappa}' \in ((\mathbf{B}_2')_{\kappa}\setminus (\mathbf{B}_2)_{\kappa})\cap \partial (B_{R_2+s})_{\kappa}(C_{\kappa})\}|\times \frac{1}{s_{\kappa}(R_2+s)}
   \end{align}
    Since
    \begin{align*}
        s_{\kappa}(R_2-s)=s_{\kappa}(R_2+s),
    \end{align*}
    and the set 
    \begin{align*}
        \{S_{\kappa}| S_{\kappa} \in ((\mathbf{B}_2)_{\kappa}\setminus (\mathbf{B}_2')_{\kappa})\cap \partial (B_{R_2-s})_{\kappa}(C_{\kappa})\}
    \end{align*}
    is the image of the antipodal map in $\mathbb{S}_{\kappa}^m$ of
    \begin{align*}
        \{S_{\kappa}'| S_{\kappa}' \in ((\mathbf{B}_2')_{\kappa}\setminus (\mathbf{B}_2)_{\kappa})\cap \partial (B_{R_2+s})_{\kappa}(C_{\kappa})\},
    \end{align*}
    the right hand sides of (\ref{measure1}) and (\ref{measure2}) are equal. Thus our desired inequality is obtained. \qed
\end{proof}

\begin{lemma}We have
\begin{align*}
    \omega(R_2-s) < \omega(R_2+s)
\end{align*}
for $0<s<R_2$.
\end{lemma}
\begin{proof}
    We begin with $M=\mathbb{R}P^n, \mathbb{C}P^n, \mathbb{H}P^n, \mathbb{O}P^2$, which are CROSS except for $\mathbb{S}^m$. Then $s<R_2<\frac{\pi}{4}$. We have two observations of the density function $\omega(t)=(\sin{t})^{m-1}(\cos{t})^{k-1}$ :
\begin{align*}\begin{cases}
    \omega'(t)>0 &\text{if}\enspace t<\arctan{\sqrt{\frac{m-1}{k-1}}},\\
     \omega'(t)<0 &\text{if}\enspace t>\arctan{\sqrt{\frac{m-1}{k-1}}}.
\end{cases}
\end{align*}
and 
\begin{align*}
    \omega(t) \le \omega(\frac{\pi}{2}-t)
\end{align*}
for $t<\frac{\pi}{4}$.
  The second observation follows from
 \begin{align*}
     \omega(\frac{\pi}{2}-t)-\omega(t)
     &=(\cos{t})^{m-1}(\sin{t})^{k-1}-(\sin{t})^{m-1}(\cos{t})^{k-1}\\
     &=(\sin{t})^{k-1}(\cos{t})^{k-1}((\cos{t})^{m-k}-(\sin{t})^{m-k})>0.
 \end{align*}
   Therefore if 
   \begin{align*}
       R_2+s<\arctan{\sqrt{\frac{m-1}{k-1}}},
   \end{align*} the first observation implies 
   \begin{align*}
       \omega(R_2-s)<\omega(R_2+s).
   \end{align*} Otherwise, the two observations give
   \begin{align*}
      \omega(R_2-s)<\omega(\frac{\pi}{2}-(R_2-s))< \omega(R_2+s).
   \end{align*} 
   Therefore the proof for CROSS follows except for $\mathbb{S}^m$. The same proof also works for $\mathbb{S}^m$ if we replace $\frac{\pi}{4}$ and $\frac{\pi}{2}$ by $\frac{\pi}{2}$ and $\pi$, respectively. \qed
\end{proof}

\section*{Acknowledgement}
The author wishes to express his gratitude to Jaigyoung Choe for helpful discussions. This research was partially supported by NRF-2018R1A2B6004262 and NRF-2020R1A4A3079066. 


%
 \section*{Conflict of interest}
The authors declare that they have no conflict of interest.



\end{document}